\documentclass[12pt]{article}

\usepackage{amsmath, amssymb, amsthm}
\usepackage{epsfig}
\usepackage{geometry}
\usepackage{amsfonts}
\usepackage{enumitem}
\usepackage{bbm}
\usepackage[font=small,labelfont=bf, width=.7\textwidth]{caption}

\begin{document}

\setcounter{page}{1} \setcounter{section}{0}
\newtheorem{theorem}{Theorem}
\newtheorem{lemma}[theorem]{Lemma}
\newtheorem{corollary}[theorem]{Corollary}
\newtheorem{proposition}[theorem]{Proposition}
\newtheorem{observation}[theorem]{Observation}
\newtheorem{definition}[theorem]{Definition}
\newtheorem{claim}{Claim}
\newtheorem{conjecture}[theorem]{Conjecture}
\newtheorem{problem}[theorem]{Problem}

\title{On Galvin orientations of line graphs and list-edge-colouring}
\author{Jessica McDonald\thanks{Department of Mathematics and Statistics, Auburn University, Auburn, AL, USA 36849, mcdonald@auburn.edu.}
}

\date{}

\maketitle

\begin{abstract} The notion of a Galvin orientation of a line graph is introduced, generalizing the idea used by Galvin in his landmark proof of the list-edge-colouring conjecture for bipartite graphs. If $L(G)$ has a proper Galvin orientation with respect to $k$, then it immediately implies that $G$ is $k$-list-edge-colourable, but the converse is not true. The stronger property is studied in graphs of the form `bipartite plus an edge', the Petersen graph, cliques, and simple graphs without odd cycles of length 5 or longer.
\end{abstract}

\section{Introduction}

In this paper we mainly consider simple graphs, however unless otherwise stated, a graph is permitted to have multiple edges (but no loops). An \emph{edge-colouring} of a graph $G$ is an assignment of colours to the edges of $G$ such that adjacent edges receive different colours; if the colours are from $\{1, 2, \ldots, k\}$ then we say it is a $k$-edge-colouring.  A graph $G$ is \emph{$k$-list-edge-colourable} if, given any set of lists $\{S_e: e\in E(G)\}$ with $|S_e|\geq k$ for all $e\in E(G)$, there is an edge-colouring of $G$ such that, for every edge $e\in E(G)$, the colour assigned to $e$ is from the list $S_e$.
The \emph{chromatic index} of $G$, denoted $\chi'(G)$, is the smallest $k$ for which $G$ is $k$-edge-colourable; the \emph{list chromatic index} of $G$, denoted $\chi'_l(G)$, is the smallest $k$ for which $G$ is $k$-list-edge-colourable. It is clear that $\chi'_l(G)\geq \chi'(G) \geq \Delta$, where $\Delta=\Delta(G)$ is the maximum degree of $G$. The list-edge-colouring conjecture (suggested by various authors, see e.g. \cite{SSTF}) postulates that $\chi'_l(G)=\chi'(G)$ always. Arguably the most important result towards the conjecture is due to Galvin \cite{Gal}, who proved that it holds for bipartite graphs.

We follow the convention of Galvin \cite{Gal} and define line graph in a way that is convenient here, although not completely standard. The \emph{line graph} of $G$, denoted $L(G)$, is the graph with vertex set $E(G)$ and where $e, f\in E(G)$ are joined by a number of edges equal to the number of ends they share in $G$ (so 2 if they are parallel, 1 if they are adjacent but not parallel, and 0 otherwise). A \emph{kernel} in a digraph is an independent set $K$ such that every vertex outside $K$ has at least one edge into $K$. A digraph is \emph{kernel-perfect} if every induced subdigraph has a kernel.

\begin{lemma}\label{kp}\emph{(Bondy, Boppana, Siegel)} Let $G$ be a graph and suppose $L(G)$ has a kernel-perfect orientation where $d^+(e) \leq k-1$ for all $e\in V(L(G))=E(G)$. Then $G$ is $k$-list-edge-colourable.
\end{lemma}

In \cite{Gal}, Galvin defined a particular orientation of the line graph of a bipartite graph, and showed that it was both kernel-perfect and had outdegree at most $\Delta$-1, hence establishing his famous result. In this paper we observe that a \emph{Galvin orientation} of $L(G)$ has a natural definition even when $G$ is not bipartite. We will see that if $L(G)$ has a \emph{proper Galvin orientation with respect to $k$} then it immediately implies that $G$ is $k$-list-edge-colourable (but that the converse is not true). In the following section we introduce these definitions and demonstrate the implication.

Although having a proper Galvin orientation with respect to $k$ is \emph{stronger} than beng $k$-list-edge-colourable, if such an orientation exists then providing one can be quite straightforward compared to other techniques for proving $k$-list-edge-colourability. This is true for graphs of the form `bipartite plus an edge'; as a first example in Section 2 we show that for such a $G$, $L(G)$ has a proper Galvin orientation with respect to $\chi'(G)$. Our proof is simpler than the proof by Plantholt and Tipnis \cite{PT} that such graphs are $\chi'$-list-edge-colourable.

It certainly not true that for every graph $G$, $L(G)$ has a proper Galvin orientation with respect to $\chi'(G)$; we will see in section 3 that this fails, in particular, when $G$ is any clique of order 4 or more. In contrast, the often-troublesome Peterson graph $\mathcal{P}$ is well-behaved; we show that $L(\mathcal{P})$ has a proper Galvin orientation with respect to $\chi'(\mathcal{P})=4$.

Given a graph $G$ where $L(G)$ does not have a proper Galvin orientation with respect to $k=\chi'(G)$, we want to know the smallest $k$ for which such an orientation exists. In Section 3 we provide a general upper bound on this $k$ for $G$ any clique. While this bound is much higher than currently known list-edge-colouring results for cliques, the gap cannot be improved in general, as it is best possible for $G=K_4$.

In our final section we focus on simple graphs without odd cycles of length 5 or longer. Such graphs (aside from $K_3$) are known to be $\Delta$-list-edge-colourable by Peterson and Woodall \cite{PW} \cite{PWe}, and hence satisfy the list-edge-colouring conjecture. Here we prove that their line graphs have proper Galvin orientations with respect to $\Delta+1$, but not $\Delta$. An above-mentioned result shows that $K_4$ has no proper Galvin orientation with respect to $\Delta$; we provide a $K_4$-free example as well.

\section{Galvin orientations of line graphs}

Given a graph $G$, a partition of $V(G)$ into sets $U$ and $D$, and a $k$-edge-colouring $\varphi$ of $G$, we define the \emph{Galvin orientation} of $L(G)$  (with respect to $U, D, \varphi$) in the following way. An edge $e \in E(L(G))$ with ends $e_1$ and $e_2$ corresponds to a common incidence of $e_1$ and $e_2$ in $G$; say $\varphi(e_1)<\varphi(e_2)$. If this common incidence is to a vertex in $D$, then we orient $e$ ``down'' in $L(G)$, from $e_2$ to $e_1$; if this common incidence is to a vertex in $U$, then we orient $e$ ``up'' in $L(G)$, from $e_1$ to $e_2$. A Galvin orientation of $L(G)$  (with respect to $U, D$, and a $k$-edge-colouring $\varphi$) is \emph{proper} if it is kernel-perfect and has outdegree at most $k-1$. By Lemma \ref{kp}, the existence of a proper Galvin orientation of $L(G)$ with respect to a $k$-edge-colouring (i.e., \emph{with respect to $k$}) proves that $G$ is $k$-list-edge-colourable.  (In fact, since we are using Lemma \ref{kp}, a proper Galvin orientation of $L(G)$ with respect to $k$ also implies that $G$ is \emph{$k$-edge-paintable}; paintability is a game-theoretic generalization of choosibility introduced by Schauz in \cite{SchPaint}).

Suppose that $G$ is bipartite, and consider the Galvin orientation of $L(G)$ with respect to any partition $(U, D)$ of $V(G)$ and any $\Delta$-edge-colouring of $G$. Given $e\in E(G)$, its end in $D$ contributes at most $\varphi(e)-1$ to its outdegree in $L(G)$, and its end in $U$ contributes at most $\Delta-(\varphi(e))$ to its outdegree in $L(G)$. Hence the outdegree of $e$ in $L(G)$ is at most $(\varphi(e)-1)+(\Delta-\varphi(e))=\Delta-1$. In \cite{Gal}, Galvin showed that this orientation is kernel-perfect (so it a proper Galvin orientation, according to our terminology), which completed his proof of the $\Delta$-list-edge-colourability of bipartite graphs. Kernel-perfectness of line graphs has since been characterized as follows. Here, a \emph{pseudochord} of a directed cycle $v_1, \ldots, v_t$ is a directed edge $v_i v_{i-1}$ for some $i$).

\begin{theorem}\label{kpchar} \emph{(Borodin, Kostochka, Woodall \cite{BKW2})} An orientation of a line graph is kernel-perfect iff every clique has a kernel and every directed odd cycle has a chord or pseduocord.
\end{theorem}

For a Galvin orientation of $L(G)$ where $G$ is bipartite and $(U, D)$ is a bipartition, $L(G)$ has no odd cycles, and every clique in $L(G)$ corresponds to a set of edges with a common incidence. If this common incidence is to a vertex in $D$, then the lowest-coloured edge in the set is a sink (and hence a kernel), and if this common incidence is to a vertex in $U$, then the highest-coloured edge in the set is a sink (and hence a kernel). In this way, such a Galvin orientation is always proper.

Consider now a not-necessarily bipartite graph $G$, and a Galvin orientation of $L(G)$ (with respect to $U, D$ and $k$-edge-colouring $\varphi$), and suppose we want to show that this orientation is proper.
Edges of $G$ with one endpoint each in $U$ and $D$ correspond to vertices with outdegree at most $k-1$, as desired. An edge $e$ with both ends in $D$ would also meet this standard if $\varphi(e)$ was a
\emph{low colour}, i.e., a colour in $\{1, 2, \ldots, \lfloor\tfrac{k+1}{2}\rfloor\}$, since each end contributes at most $\lfloor\tfrac{k+1}{2}\rfloor -1$ to outdegree, and $2(\lfloor\tfrac{k+1}{2}\rfloor -1) \leq 2(\tfrac{k+1}{2}-1)=k-1$. Similarly, an edge $e$ with both ends in $U$ would also meet the standard if $\varphi(e)$ was a \emph{high colour}, i.e., a colour in $\{\lceil\tfrac{k+1}{2}\rceil, \ldots, k\}$, since each end contributes at most $k- \lceil \tfrac{k+1}{2}\rceil$ to outdegree, and $2(k-\lceil \tfrac{k+1}{2}\rceil) \leq 2(k-\tfrac{k+1}{2})=k-1$. If both ends of $e$ have degree $\Delta$ and $k=\Delta$, then in fact this is the only way to meet the outdegree standard -- that is, if $e$ is between two $U$'s then it must have a high colour, and if $e$ is between two $D$'s then it must have a low colour. However, in other situations, it may happen that the outdegree condition is met without this.

Cliques in $L(G)$ either correspond to a set of edges with a common incidence, in which case they have a kernel (as in the bipartite case), or to a triangle, which must be transitive as opposed to directed in order to have a kernel. Suppose we have a triangle in $G$ with three edges $e_1, e_2, e_3$ and $\varphi(e_1)<\varphi(e_2)<\varphi(e_3)$. If both ends of $e_1$ are in $D$ then $e_1$ is a sink in the orientation of the corresponding triangle in $L(G)$ (since $\varphi(e_1)$ is the lowest-coloured edge in the triangle), and if both ends of $e_1$ are in $U$ then $e_1$ is a source (again, since $\varphi(e_1)$ is the lowest-coloured edge in the triangle). Similarly, if both ends of $e_3$ are in the same partite set, then $e_3$ is either a source (if both ends are in $D$) or a sink  (if both ends are in $U$), since $e_3$ is the highest-coloured edge in the triangle. Note that having the ends of $e_2$ in different partite sets is equivalent to having exactly one of $e_1, e_3$ with ends in the same partite set. To finish checking kernel-perfectness (and hence complete our proof that the Galvin orientation is proper), we would also have to know that all directed odd cycles of length 5 or longer in $L(G)$ have a chord or pseudocord.

As a first example of proper Galvin orientations, we provide the following result about graphs of the form `bipartite plus an edge'. Plantholt and Tipnis \cite{PT} proved that such graphs satisfy the list-colouring conjecture, however the following proof turns out to be simpler.

\begin{theorem}\label{bipplusone} If $G$ is obtained from a bipartite graph by adding an edge, then $L(G)$ has a proper Galvin orientation with respect to $\chi'(G)$. The same conclusion is obtained if a multi-edge is added, provided its multiplicity is no more than $\lfloor\tfrac{\chi'(G)+1}{2}\rfloor$.
\end{theorem}

\begin{proof} If $G$ remains biparitite, then the result is true by Galvin \cite{Gal}, as discussed above. Otherwise, partition the vertices of $G$ into sets $D$ and $U$ so that $G[U]$ is independent and $G[D]$ contains only one edge $e$ (and possibly some edges in parallel to $e$). Assign a $\chi'(G)$-edge-colouring to $G$ so that $e$, and any edges in parallel to $e$, receive the lowest colours, i.e. colours $1, 2, \ldots$ (as many colours as needed). By assumption, these are all low colours, and hence the Galvin orientation satisfies the outdegree condition. Moreover, no cycle containing $e$ (or an edge in parallel to $e$) is directed: the previous and next edge in any such cycle would both have higher colours, making $e$ (or its parallel edge) a sink in the cycle.
\end{proof}

One way to approach defining a proper Galvin orientation is to try to maximize the number of edges between $U$ and $D$, as we did in the above proof. Certainly, this can be helpful when checking the outdegree condition. Note however that a cycle all of whose vertices are in either $U$ or $D$ cannot be directed, so forcing such cycles can also be helpful. If the entire graph is an odd cycle, then in fact it has a proper Galvin orientation (with respect to $\chi'=3$) where all vertices are in $D$ (or all vertices are in $U$). This is because each edge in an odd cycle will have total degree $2$ in the line graph, so the outdegree condition is trivially met.

\section{The Petersen graph versus cliques}

The Petersen graph is known to satisfy the list-colouring conjecture (this follows from Juvan, Mohar and Skrekovski \cite{JMS}); here we show that additionally, its line graph has a proper Galvin orientation with respect to $4$.

\begin{theorem} Let $\mathcal{P}$ be the Petersen graph. Then $L(\mathcal{P})$ has a proper Galvin orientation with respect to $\chi'(\mathcal{P})=4$.
\end{theorem}

\begin{proof} We claim that a proper Galvin orientation of $L(\mathcal{P})$ with respect to 4 is explicitly given in Figure \ref{PetersenGalvin}. In the figure, a 4-edge-colouring is given, and each vertex is assigned to one of $U$ or $D$. Note that all edges between two $D$'s have low colours (colours 1 or 2) and that all edge  between two $U$'s have high colours (3 or 4). Hence the outdegree condition is satisfied.

\begin{figure}
\centering
\includegraphics[width=9cm]{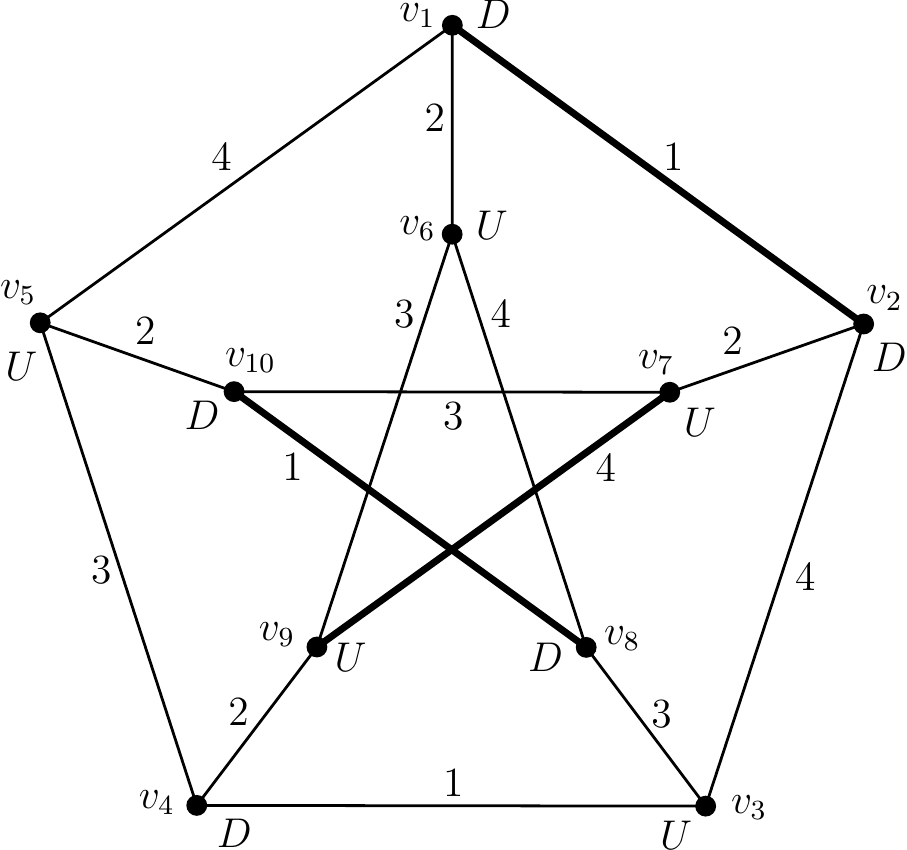}
\caption{A proper Galvin orientation of $L(\mathcal{P})$ with respect to a 4-edge-colouring.}
\label{PetersenGalvin}
\end{figure}

Three edges are bolded in the figure -- $v_1v_2, v_7v_9$, and $v_8v_{10}$ -- these are either 1-edges between two $D$'s or $4$-edges between two $U$'s. Since 1 is our lowest colour and 4 is our highest colour, these three edges will be sinks in any cycles in which they occur. Hence to check the cycle condition, it suffices to check that all odd cycles in $\mathcal{P}-\{v_1v_2, v_7v_9, v_8v_{10}\}$ will be undirected in the line graph. The only odd cycles in this graph are the following two 5-cycles: $v_1, v_6, v_9, v_4, v_5$ and $v_3, v_4, v_9, v_6, v_8$. In the first of these, the edge $v_6v_9$ is a sink, and in the second, the edge $v_4v_9$ is a source.
\end{proof}

Although the Petersen graph is (unusually) well-behaved, it is certainly not true that for every graph $G$, $L(G)$ has a proper Galvin orientation with respect to $\chi'(G)$. In fact, we show now that this fails when $G$ is any clique of order 4 or more. (We already know that is does not fail for $K_3$, as odd cycles are included in Theorem \ref{bipplusone}, and additionally discussed at the end of section 2).

\begin{theorem}\label{cliquecounter} $L(K_n)$ has no proper Galvin orientation with respect to $\chi'(K_n)$, for any $n\geq 4$.
\end{theorem}

\begin{proof} Suppose, for a contradiction, that $L(K_n)$ does have a proper Galvin orientation with respect to $k=\chi'(K_n)$, say with $\varphi, D, U$.

First suppose that $n$ is even. Then, every colour in $\varphi$ induces a perfect matching in $K_n$. Note not all vertices of $K_n$ can be in $U$, since the 1-edges cannot be between two $U$'s and still satisfy the outdegree condition. So, at least one vertex must be in $D$, say vertex $x$. (See Figure \ref{K4counter} for a depiction of the $n=4$ case). Let $y$ be the vertex such that $xy$ has colour $k$. Since $k$ cannot be between two $D$'s, $y$ must be in $U$. Let $z$ be the vertex such that $yz$ has colour $1$. Since 1 cannot be between two $U$'s, $z$ must be in $D$. However now, the triangle $xyz$ corresponds to a directed triangle in the line graph: its highest- and lowest-coloured edges are both between $U$ and $D$.

\begin{figure}
\centering
\includegraphics[width=6cm]{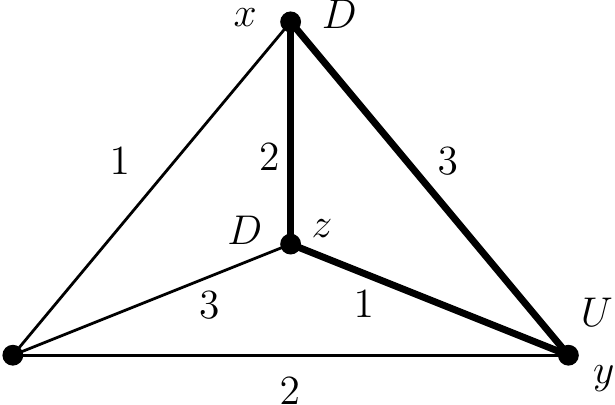}
\caption{$L(K_4)$ does not have a proper Galvin orientation with respect to $\chi'(K_4)=3$.}
\label{K4counter}
\end{figure}

Now suppose that $n$ is odd. In $\varphi$, every colour in $1, 2, \ldots, k$ is missing at exactly one vertex of $K_n$. Since $n\geq 5$, every colour occurs on at least two edges. If a 1-edge is between two $U$'s, then even though $k=\Delta+1$ (as opposed to $\Delta$), the corresponding outdegree will be $2(\Delta-1)\geq k=\Delta+1$ (since $\Delta\geq 4$), and the outdegree condition will be violated. So a $1$-edge cannot be between two $U$'s, and hence there are at least two vertices of $K_n$ in $D$. Choose one of these, say $x$, that is incident to an edge of colour $k$, and let $y$ be the other end of the $k$-edge. For the same reason that 1 cannot be between two $U$'s, $k$ cannot be between two $D$'s, so $y$ must be in $U$. If $1$ is not missing at $y$, then taking $z$ where $yz$ is this 1-edge, we get that $xyz$ corresponds to a directed triangle in the line graph, just as in the previous case. So, suppose that $1$ is missing at $y$. See Figure \ref{oddcliquecounter}. Then 2 is not missing at $y$, so we may choose $z'$ such that $yz'$ is a 2-edge. If $z'$ were in $U$, the outdegree of this 2-edge would be $(\Delta-1)+(\Delta-2)$ ($\Delta-1$ from $y$, where no 1 is present). Since $\Delta\geq 4$, $2\Delta-3 \geq k=\Delta+1$ , so this would violate the outdegree condition. Hence $z' \in D$.

\begin{figure}
\centering
\includegraphics[width=6.5cm]{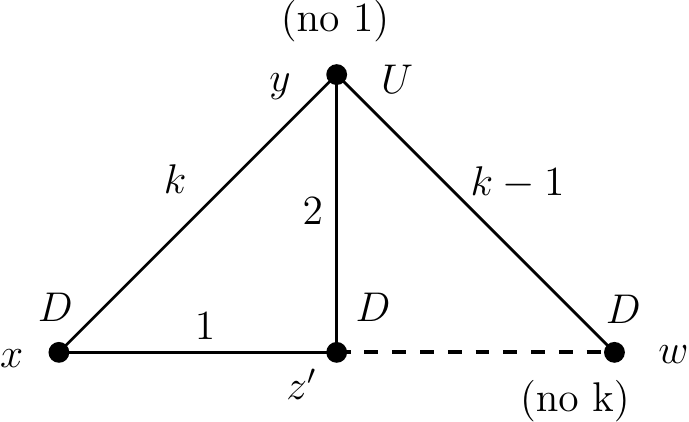}
\caption{The vertices $x, y, z', w$ in the odd clique case for the proof of Theorem \ref{cliquecounter}.}
\label{oddcliquecounter}
\end{figure}

The only way that $xyz'$ will not yield a directed triangle is if $xz'$ has a colour lower than 2, i.e., if $xz'$ has colour $1$, so assume this is so. We previously argued that each 1-edge has at least one end that is in $D$, and there are at least two 1-edges. Since we have just seen that the 1-edge $xz'$ has both ends in $D$, this means that there are at least 3 vertices in $D$; let $w\neq x, z'$ be another vertex in $D$. If $w$ is not missing colour $k$, then 1 will not be missing at the other end of this $k$-edge (since 1 is only missing at one vertex), and hence we will get a directed triangle, as previously. So, $w$ must be missing colour $k$. Then, however, it is not missing colour $k-1$. Let $y'$ be the other end of this $(k-1)$-edge. For the same reason that $yz'$ (coloured 2 and with 1 missing missing at $y$) cannot be between two $U's$, $wy'$ (coloured $k-1$ and with $k$ missing at $w$) cannot be between two $D$'s, so $y'\in U$.

We claim that $y'=y$. If not, then there is a 1-edge incident to $y'$, with other end $z''\in D$. Since $w$ is missing $k$, the colour of edge $wz''$ must be be more than 1 and less than $k-1$, and hence the triangle $zy'z''$ yields a directed triangle in the line graph. So, indeed, $y'=y$.

Finally, consider the edge $wz'$ (depicted as dotted in Figure \ref{oddcliquecounter}). This edge cannot have colour 1, since $z'x$ does, nor can it have colour $k$, since $k$ is missing at $w$. So it has a colour that is larger than $2$ and smaller than $k-1$, and hence $ywz' $ yields a directed triangle in the line graph.
\end{proof}

While $L(K_4)$ has no proper Galvin orientation with respect to $3$, it it easy to see that does have such an orientation with respect to $4$. We generalize this to the following result providing proper Galvin orientations for all cliques.

\begin{theorem}\label{cliquebound} $L(K_{\Delta+1})$ has a proper Galvin orientation with respect to $f(\Delta+1)$, where
$$f(\Delta+1)= \left\lfloor\frac{3\Delta}{2}\right\rfloor +\begin{cases}
      0 & \textrm{if $\Delta+1\equiv 0$ mod 4} \\
      1 & \textrm{if $\Delta+1\equiv 1$ mod 4} \\
      2 & \textrm{if $\Delta+1\equiv 2$ mod 4} \\
      1 & \textrm{if $\Delta+1\equiv 3$ mod 4}.
   \end{cases}$$
\end{theorem}

\begin{proof}
Split the vertices of $K_{\Delta+1}$ into sets $D$ and $U$ as evenly as possible, with $|D|\leq |U|$.  The graph $K_{\Delta+1}$ is thus partitioned into two cliques $G[D]$, $G[U]$, and a complete bipartite graph $G[D, U]$. Edge-colour $K_{\Delta+1}$ as follows: use colours $1, 2, \ldots, \chi'(G[D])$ on $G[D]$, use colours $\chi'(G[D])+1,\ldots, \chi'(G[D])+\chi'(G[U, D])$ on $G[D, U]$, and use colours $\chi'(G[D])+\chi'(G[U, D])+ 1, \ldots, \chi'(G[D])+\chi'(G[U, D])+ \chi'(G[U])$ on $G[U]$.

The total number of colours used on $K_{\Delta+1}$ depends on the congruence class of $\Delta+1$ modulo 4. If the congruence class is 0, then $G[D]$ and $G[U]$ are both even cliques of size $\tfrac{\Delta+1}{2}$, so
$$\chi'(G[D])+\chi'(G[U, D])+ \chi'(G[U])= (\tfrac{\Delta+1}{2}-1)+ (\tfrac{\Delta+1}{2})+ (\tfrac{\Delta+1}{2}-1)=\tfrac{3\Delta-1}{2}=
\left\lfloor\tfrac{3\Delta}{2}\right\rfloor.$$
If the congruence class is 1, then $G[D]$ is an even clique of size $\tfrac{\Delta}{2}$ and $G[U]$ is an odd clique of size $\tfrac{\Delta}{2}+1$, so
$$\chi'(G[D])+\chi'(G[U, D])+ \chi'(G[U])= (\tfrac{\Delta}{2}-1)+ (\tfrac{\Delta}{2}+1)+ (\tfrac{\Delta}{2}+1)=\tfrac{3\Delta}{2}+1
=\left\lfloor\tfrac{3\Delta}{2}\right\rfloor+1.$$
If the congruence class is 2, then $G[D]$ and $G[U]$ are both odd cliques of size $\tfrac{\Delta+1}{2}$, so
$$\chi'(G[D])+\chi'(G[U, D])+ \chi'(G[U])= (\tfrac{\Delta+1}{2})+ (\tfrac{\Delta+1}{2})+ (\tfrac{\Delta+1}{2})=\tfrac{3\Delta+3}{2}=
\left\lfloor\tfrac{3\Delta}{2}\right\rfloor+2.$$
If the congruence class is 3, then $G[D]$ is an odd clique of size $\tfrac{\Delta}{2}$ and $G[U]$ is an even clique of size $\tfrac{\Delta}{2}+1$, so
$$\chi'(G[D])+\chi'(G[U, D])+ \chi'(G[U])= (\tfrac{\Delta}{2})+ (\tfrac{\Delta}{2}+1)+ (\tfrac{\Delta}{2})=\tfrac{3\Delta}{2}+1
=\left\lfloor\tfrac{3\Delta}{2}\right\rfloor+1.$$
In all cases, we see that $K_{\Delta+1}$ has been edge-coloured with $f(\Delta+1)$ colours.

The colours used on $G[D]$ are all low with respect to $f(\Delta+1)$, and the colours used on $G[U]$ are all high with respect to $f(\Delta+1)$, hence the the outdegree requirement is satisfied. A triangle in $G$ with at least two vertices in $D$ satisfies the condition that the lowest-coloured edge of the triangle occurs between two vertices in $D$. Any other triangle has two or more vertices in $U$ and hence satisfies the condition that the highest-coloured edge of the triangle occurs between two vertices in $U$. So, the orientation has no directed triangles. In fact, we claim it has no directed odd cycles at all. Odd cycles in $G[U]$ and $G[D]$ are obviously not directed. Moreover, any edge in $G[U,D]$ will be lower-coloured and hence point towards any edge in $G[U]$; any edge in $G[U, D]$ will be higher-coloured and hence point towards any edge in $G[D]$. So, indeed the orientation of $L(G)$ will not contain any directed odd cycles.
\end{proof}

For $K_4$, the above provides a proper Galvin orientation of $L(K_4)$ with respect to $f(4)=4$, which is best possible by Theorem \ref{cliquecounter}. Despite this, the gap is large between the bound of Theorem \ref{cliquecounter} and the best-known list-edge-colouring results for cliques. Borodin, Kostochka and Woodall \cite{BKW1} have proved that
$\chi'_l(G)\leq \lfloor\tfrac{3\Delta}{2}\rfloor$ for all graphs $G$; this bound is attained by Theorem \ref{cliquebound} when the order of the clique is divisible by 4, but otherwise the bound of Theorem \ref{cliquebound} is higher. Moreover, $\lfloor\tfrac{3\Delta}{2}\rfloor$ is far from the best known list-edge-colouring bound for cliques: H\"{a}ggkvist and Janssen proved that $\chi'_l(K_n)\leq n$, establishing the list-edge-colouring conjecture for odd cliques, and providing a bound that is only one away from the conjecture even cliques. Recently Schauz \cite{Sch} further added to this by proving the conjecture for $K_{p+1}$ where $p$ is an odd prime.

\section{Simple graphs without odd cycles of length 5 or longer}

A \emph{block} in a graph is a maximal connected subgraph that has no cut vertices (vertices whose removal disconnects the graph).

\begin{theorem}\label{blocks} \emph{(Maffray \cite{Maf})} A simple graph $G$ has no odd cycles of length 5 or more iff every block $B$ of $G$ satisfies one of the following:
\begin{enumerate}
\item  $B$ is bipartite;
\item $B=K_4$; or
\item $B$ consists of exactly $p+2$ vertices $x_1, x_2, \ldots, x_p, v, w$ where $\{x_1, \ldots, x_p\}$ is an independent set and $\{x_i, v, w\}$ induces $K_3$ for each $i\in\{1, 2, \ldots, p\}$.
\end{enumerate}
\end{theorem}

For our purposes we will call a block $B$ of the third category above a block of \emph{type $T$}. We call the edge $vw$ the \emph{spine} of such a block, and we call each vertex in $\{x_1, \ldots, x_p\}$ a \emph{spike} of the block.

\begin{theorem}\label{main} Let $G$ be a simple graph without odd cycles of length 5 or longer. Then $L(G)$ has a proper Galvin orientation with respect to $\Delta+1$, where $\Delta$ is the maximum degree of $G$.
\end{theorem}

\begin{proof} Let $k=\Delta+1$. We use the characterization of Theorem \ref{blocks} and define an appropriate $k$-edge-colouring $\varphi$ of $G$, and appropriate partition $U, D$ of $V(G)$, one block at a time. At each step we choose a block $B$ of $G$  which has at most one vertex in common with our previously chosen blocks. We assume that $b\in V(B)$  has already been assigned to one of $U, D$ (if not, i.e., for the first block, make this assignment arbitrarily), and assume that all edges incident to $b$ but not in $B$ have already been coloured (if not, it would only help us to have more colours available). The number of colours available for edges incident to $b$ is at least $deg_B(b)+1$, since $k=\Delta+1$. 

\underline{If $B$ is bipartite:} Fix a bipartition $B_D, B_U$ of $B$ with $b\in B_D$ iff $b\in D$, and put all vertices of $B_D$ in $D$ and all vertices of $B_U$ in $U$. Assign a proper $k$-edge-colouring of $B$ arbitrarily, except for the restriction that the colouring must be proper at $b$.  Since every edge in $B$ goes between $D$ and $U$, and since $B$ has no triangles,  there is nothing to check here.

\underline{If $B=K_4$:} Suppose first that $b\in D$ and that there are three (or more) available low colours at $b$. Then we can 3-edge-colour $B$ with these three colours, and add every vertex in $B$ to $D$. See the leftmost graph in Figure \ref{K4case}. Analogously, if $b\in U$ and there are three (or more) available high colours at $b$, we can 3-edge-colour $B$ with these three colours and add every vertex to $U$.

Suppose now that $b\in D$ and there are (at least) four available high colours at $b$. Let $\alpha< \beta<\gamma<\tau$ be such colours. Colour the edges incident to $b$ with $\alpha, \beta, \gamma$, colour the edge between the $\alpha$-edge and the $\gamma$-edge with $\tau$, and colour the other two edges with $\gamma$ and $\alpha$. Assign all three vertices to $U$. See the middle graph in Figure \ref{K4case}. The only edges with both ends in the same partite set have both ends in $U$, and these edges all have high colours. Every triangle in $B$ contains either the $\tau$-edge or the $\beta$-edge. The colour $\tau$ is the highest colour used in $B$ and is between two vertices in the same partite set. The two triangles containing $\beta$ have other colours $\alpha$ and $\gamma$, so $\beta$ is neither the highest or lowest-coloured edge in these triangles, and it has ends in different partite sets. Analogously, if $b\in U$ and there are (at least) four available low colours $\alpha>\beta>\gamma>\tau$ at $b$, then we satisfy our required properties by giving the same edge-colouring to $B$ as just described, but assigning all three vertices to $D$.

\begin{figure}
\centering
\includegraphics[width=\textwidth]{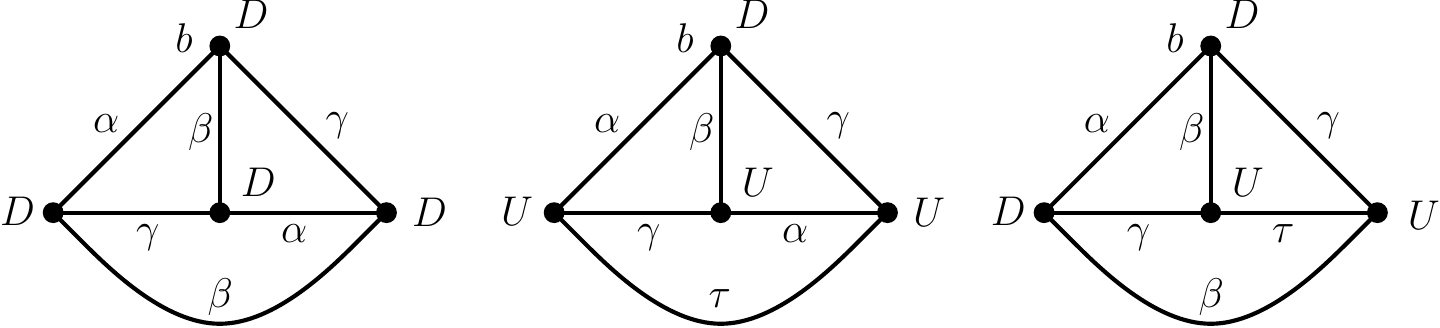}
\caption{The $B=K_4$ case.}
\label{K4case}
\end{figure}

We may now assume that if $b\in D$, there are at most 3 available high colours at $b$, and at most two available low colours at $b$. This means that we have $\alpha<\beta<\gamma<\tau$  available, and $\alpha$ is low and $\gamma, \tau$ are high.  Colour the edges incident to $b$ with $\alpha, \beta, \gamma$, colour the edge between the $\beta$-edge and $\gamma$-edge with $\tau$, and colour the other two edges with $\beta$ and $\gamma$.  Assign both ends of the $\tau$-edge to $U$, and assign the other vertex to $D$. See the rightmost graph in Figure \ref{K4case}. The only edges between two ends in the same partite set are the $\tau$-edge (which is high and between two $U$'s), and the $\alpha$-edge (which is low and between to $D$'s). Every triangle in $B$ contains either the $\tau$-edge or the $\alpha$-edge. The colour $\tau$ is the highest colour used in $B$ and is between two vertices in the same partite set; the colour $\alpha$ is the lowest colour used in $B$ and is between two vertices in the same partite set. Analogously, if $b\in U$ then we may now assume that we have $\alpha>\beta>\gamma>\tau$ available at $b$ where $\alpha$ is high, and $\gamma, \tau$ are low. We again satisfy our required properties by giving the same edge-colouring as above, but assigning both ends of the $\tau$ edge to $D$ assigning the other vertex to $U$.

\underline{If $B$ is of type $T$ and $b$ is a spike:} Assume first that $b\in D$, and let $\alpha<\beta<\gamma$ be three available colours at $b$. Colour the two edges incident to $b$ with $\alpha, \beta$, colour the spine with $k$, assign its two ends to $U$, and assign all other spikes to $D$. See the leftmost graph in Figure \ref{TypeTcase}. Regardless of how the remaining edges of $B$ are (properly) coloured, the only edge between two vertices in the same partite set is our $k$-edge, and it is between two $U$ vertices. Moreover, every triangle in $B$ contains this $k$-edge between two $U$ vertices. Analogously, if $b\in U$ and $\alpha>\beta>\gamma$ are three available colours at $b$, then we satisfy our required properties by colouring the two edges incident to $b$ with colours $\alpha, \beta$, colouring the spine with $1$, assigning both ends of the spine to $D$, and assigning all other spikes to $U$.

\begin{figure}
\centering
\includegraphics[width=\textwidth]{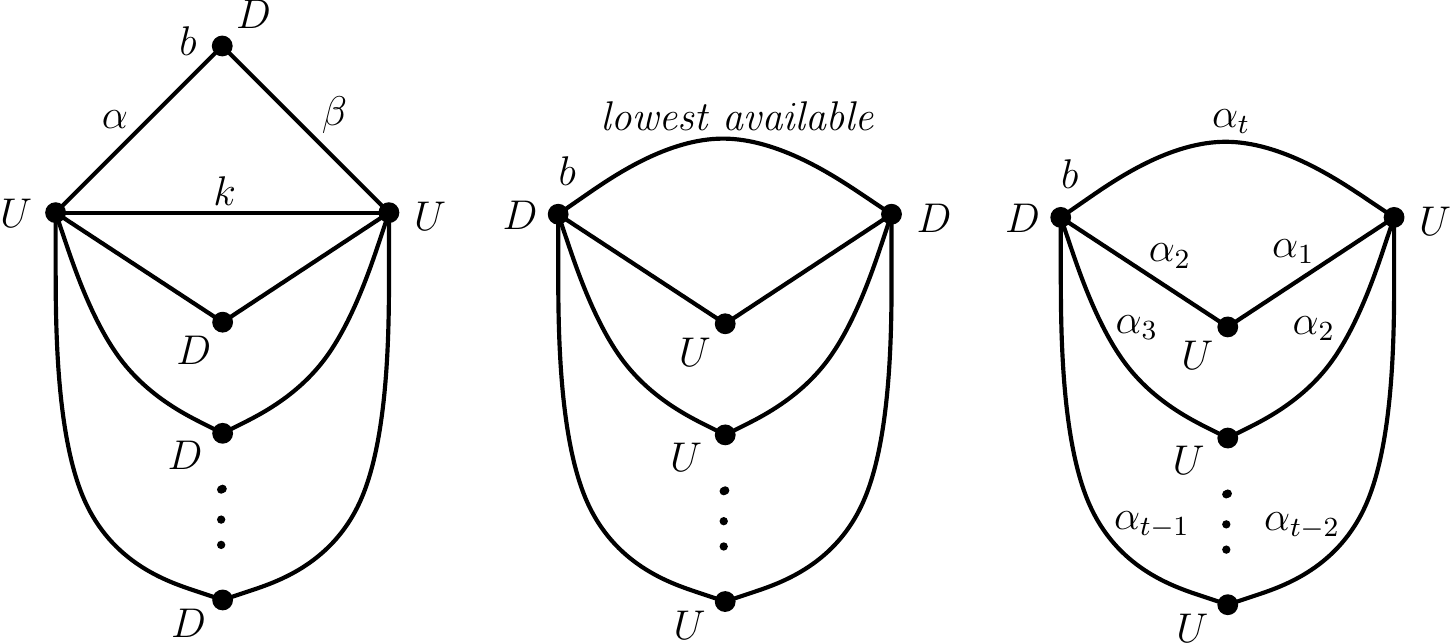}
\caption{The case when $B$ is of type $T$.}
\label{TypeTcase}
\end{figure}

\underline{If $B$ is of type $T$ and $b$ is not a spike:} Suppose first that $b\in D$ and that a low colour is available at $b$. Then, colour the spine with the lowest such colour available, colour the other edges incident to $b$ arbitrarily with other available colours, assign all the spikes to $U$ and assign the other end of the spine to $D$. Colour the remaining edges either by permuting the colours used on the non-spine edges incident to $b$, or, if there is only one spike, use an additional available colour at $b$ on the remaining edge (such a colour exists because $k>\Delta$). See the middle graph in Figure \ref{TypeTcase}. The only edge with both ends in the same partite set is the spine, and it is a low colour between two $D$'s. Every triangle in $B$ contains this spine, which is the lowest colour used on $B$ and is between two vertices in the same partite set. Analogously, if $b\in U$ and a high colour is available at $b$, then we satisfy our required properties by using the highest such colour on the spine, adding the other end of the spine to $U$ and all the spikes to $D$, and colouring the other edges of $b$ arbitrarily (but properly) using the other available colours at $b$.

Now suppose that $b\in D$ and that no low colours are available at $b$. Suppose that $deg_B(b)=t-1$ and among the colours available at $b$ are $\alpha_1<\alpha_2 <\cdots <\alpha_t$. Add all the vertices in $B$ (besides $b$) to $U$. Colour the spine with $\alpha_t$, and colour the other edges incident to $b$ arbitrarily using only the colours $\alpha_2 < \ldots <\alpha_{t-1}$ (note that $\alpha_1$ is not needed). For the remaining edge in $B$ incident to colour $\alpha_i$ (for $i\in\{2, 3, \ldots, t-1\}$), colour it with colour $\alpha_{i-1}$. See the rightmost graph in Figure \ref{TypeTcase}. The only edges between two vertices in the same partite set are these last edges we coloured, which are all high colours between two $U$'s. In every triangle of $B$, the middle colour of the triangle is on the non-spine edge incident to $b$, and this edge has ends in different partite sets. Analogously, if $b\in U$ and available colours at $b$ include $\alpha_1>\alpha_2 \cdots >\alpha_t$, none of which are high, then we satisfy our required properties by giving the same colouring we just described, but adding all the vertices in $B$ (besides $b$) to $D$.
\end{proof}

If $G$ is a simple graph without odd cycles of length 5 or longer, then $L(G)$ is known to be perfect (by Maffray \cite{Maf}), that is $\chi'(G)=\chi(L(G))=\omega(L(G))$, where $\omega$ denotes the size of the largest clique. Furthermore, $\omega(L(G))=\Delta(G)$ unless the largest clique in $L(G)$ is a triangle which corresponds to a triangle in $G$ and $\omega(L(G))=3>\Delta(G)=2$. If $G$ is connected, then this latter situation means that $G=K_3$. So, connected simple graphs with no odd cycles of length 5 or longer all have chromatic index $\Delta$, with the exception of $K_3$. Peterson and Woodall \cite{PW}\cite{PWe} showed that such graphs are $\Delta$-list-edge-colourable.

Since $L(K_4)$ has no proper Galvin orientation with respect to $\Delta=3$ (by Theorem \ref{cliquecounter}), we know that Theorem \ref{main} is best possible in general. In fact, Theorem \ref{main} is best possible even when $G$ is assumed to be $K_4$-free.

\begin{theorem}\label{K4freecounterthm} Let $G$ be the graph depicted in Figure \ref{K4freecounter}, which is simple and contains $K_3$ but has no odd cycles of length 5 or longer, and is $K_4$-free. $L(G)$ has no proper Galvin orientation with respect to $\Delta(G)=4$.
\end{theorem}

\begin{figure}
\centering
\includegraphics[width=13.5cm]{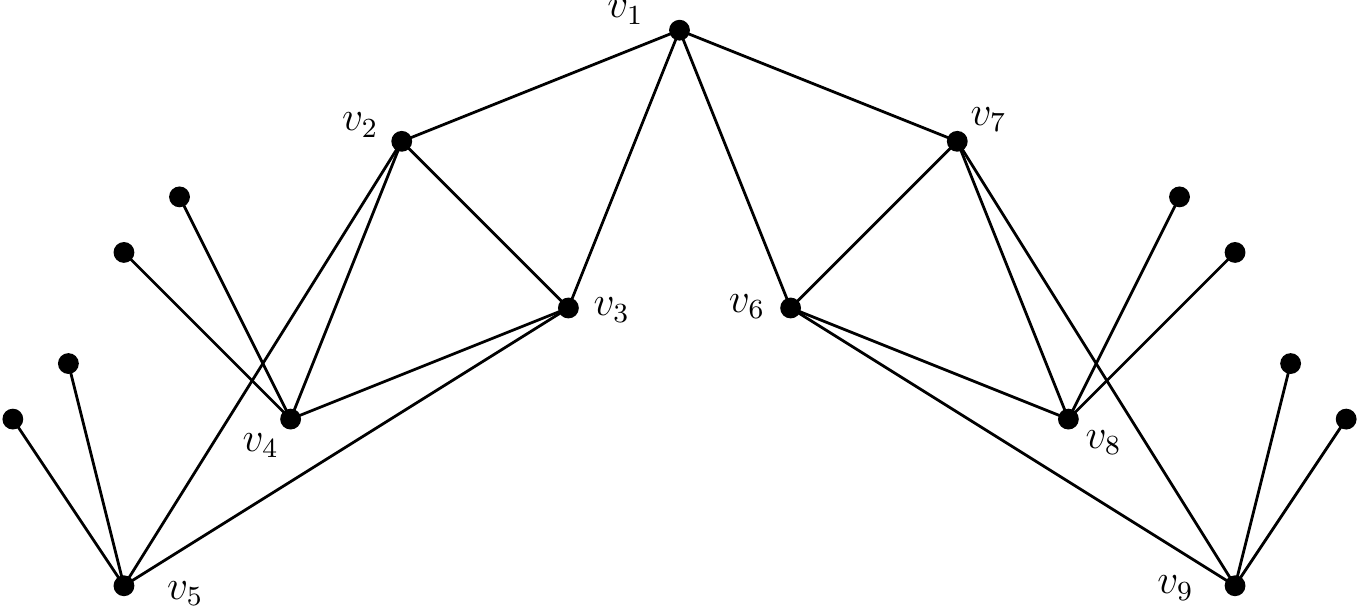}
\caption{The graph $G$ in Theorem \ref{K4freecounterthm}.}
\label{K4freecounter}
\end{figure}

\begin{proof} Suppose, for a contradiction, that $L(K_n)$ does have a proper Galvin orientation with respect to $4$, say with $\varphi, D, U$. Note that $v_1, \ldots, v_9$ are all of degree 4, so all edges with both ends in this set must follow the rule that a colour between two $U$'s is high (3 or 4), and a colour between two $D$'s is low (1 or 2).

We suppose that $v_1\in D$ -- if $v_1\in U$, the argument follows symmetrically. Since $v_1$ has degree $4$, the colour $4$ must be incident to $v_1$; without loss of generality suppose that $v_1v_2$ has colour 4. Then $v_2\in U$. We claim that $v_3\in D$. If not, then $v_2v_3$ must be a high colour (hence 3). Regardless of how the edge $v_1v_3$ is coloured (either 1 or 2), neither the highest nor lowest colour in the triangle $v_1v_2v_3$ is between two $U$'s or two $D$'s, meaning that $v_1v_2v_3$  gives a directed triangle in $L(G)$. Hence $v_3\in D$.

Since $\varphi$ is a 4-edge-colouring, either the edge $v_3v_4$ or the edge $v_3v_5$ must have colour 4; without loss of generality we assume that $v_3v_4$ has colour 4. This implies that $v_4\in U$. So $v_2v_4$ must have a high colour, hence 3. However now, regardless of how $v_2v_3$ is coloured (either colour 1 or 2), neither the highest nor lowest colour in the triangle $v_2v_3v_4$ is between two $U$'s or two $D$'s, meaning that $v_2v_3v_4$  gives a directed triangle in $L(G).$
\end{proof}

\end{document}